\documentclass[runningheads]{llncs}

\usepackage[T1]{fontenc}
\usepackage{graphicx}

\usepackage{amssymb}
\usepackage{amsmath}
\allowdisplaybreaks

\usepackage{hyperref}
\usepackage{color}

\usepackage{setspace}

\newcommand{\BR}{\mathbb{R}}
\newcommand{\BC}{\mathbb{C}}
\newcommand{\Mat}{{\rm Mat}}
\newcommand{\Det}{{\rm Det}}
\newcommand{\tr}{{\rm tr}}
\newcommand{\Adj}{{\rm Adj}}
\def\cl{\mathcal {G}}

\usepackage{cancel}

\begin{document}
%

\title{Noncommutative Vieta Theorem in Clifford Geometric Algebras
}

%
%
\author{Dmitry Shirokov\inst{1,2}\orcidID{0000-0002-2407-9601} 
}
\authorrunning{Dmitry Shirokov}
%

\institute{HSE University, 101000 Moscow, Russia\\
\email{dshirokov@hse.ru} \and
Institute for Information Transmission Problems of the Russian Academy of Sciences, 127051 Moscow, Russia\\
\email{shirokov@iitp.ru}
}


%
\maketitle              
\begin{abstract}
In this paper, we discuss a generalization of Vieta theorem (Vieta's formulas) to the case of Clifford geometric algebras. We compare the generalized Vieta's formulas with the ordinary Vieta's formulas for characteristic polynomial containing eigenvalues. We discuss Gelfand--Retakh noncommutative Vieta theorem and use it for the case of geometric algebras of small dimensions. We introduce the notion of a simple basis-free formula for a determinant in geometric algebra and prove that a formula of this type exists in the case of arbitrary dimension. Using this notion, we present and prove generalized Vieta theorem in geometric algebra of arbitrary dimension. The results can be used in symbolic computation and various applications of geometric algebras in computer science, computer graphics, computer vision, physics, and engineering.

\end{abstract}

\section{Introduction}

In algebra, Vieta's formulas (or Vieta theorem) relate the coefficients of any polynomial to sums and products of its roots. These formulas are named after the famous French mathematician François Viète (or Franciscus Vieta). In this paper, we extend Vieta's formulas to (Clifford) geometric algebras. We discuss the noncommutative Vieta theorem in geometric algebras and compare it with the ordinary Vieta theorem. 

In this paper, the notion of characteristic polynomial in geometric algebras is used. Note that the determinant is used to calculate the inverse in geometric algebras \cite{rudn,hitzer1,acus,det}.
In \cite{det,Abd2,Abd}, the explicit formulas for the characteristic polynomial coefficients $C_{(k)}$ are presented in the cases $n\leq 6$. The characteristic polynomial in geometric algebras is also discussed in \cite{Helm} and used to solve the Sylvester and Lyapunov equations in \cite{ShirokovSylv,ShirokovSylv2}. It is used to calculate exponentials and other elementary functions of multivectors in \cite{acus2}. We expect further application of the characteristic polynomial and related concepts in symbolic computation, computer science, computer graphics, computer vision, physics, and engineering using the formalism of geometric algebras.

In this paper, we discuss three different approaches to the problem of computing the characteristic polynomial coefficients: 1) generalized Vieta formulas based on the Gelfand--Retakh theorem (for $n\leq 3$), where solutions of characteristic equations are provided by multivectors (non-scalars); 2) recursive formulas based on the Faddeev--LeVerrier algorithm (for arbitrary $n$); and 3) explicit optimal formulas, which follow from the optimized form of the determinant (proved for arbitrary $n$, but explicitly known for $n\leq 6$). The first approach is discussed in Sections~\ref{sectVieta}--\ref{cl23}. The second approach is discussed at the beginning of Section~\ref{sectBasis} (and in detail in the previous paper \cite{det}) and is used to obtain the third approach in Sections~\ref{sectBasis}--\ref{sectBar}. Note that Vieta's original theorem is about roots of a one-variable polynomial equation. The Gelfand--Retakh theorem generalizes Vieta theorem to the noncommutative case (for example, for complex matrices or Clifford algebra elements). In this paper, we apply the Gelfand--Retakh theorem to a particular case of the characteristic equation in Clifford geometric algebras, whose coefficients are scalars, but the variable is not a scalar.

This paper is organized as follows. In Section \ref{sectVieta}, we introduce generalized Vieta's formulas in geometric algebras and compare them with the ordinary Vieta's formulas for characteristic polynomial containing eigenvalues. The generalized Vieta's formulas do not contain eigenvalues. In Section \ref{secn23}, we discuss Gelfand--Retakh noncommutative Vieta theorem for an arbitrary skew-field with some remarks. In Section \ref{cl23}, we apply noncommutative Vieta theorem to the geometric algebras $\cl_{p,q}$ in the case of small dimensions $n=p+q$. In Section \ref{sectBasis}, we introduce the notion of a simple basis-free formula for the determinant in geometric algebra and prove that a formula of this type exists in the case of arbitrary $n$. In Section \ref{sectGeneral}, we present and prove generalized Vieta theorem in geometric algebra $\cl_{p,q}$ with arbitrary $n=p+q$. In Section \ref{sectBar}, we discuss three alternative approaches to finding optimized formulas for the determinant and other characteristic polynomial coefficients using the bar-operation. The conclusions follow in Section \ref{sectConcl}.

This paper is an extended version of the short note in Conference Proceedings \cite{ShirokovCGI2022} (Empowering Novel Geometric Algebra for Graphics \& Engineering Workshop, Computer Graphics International 2022). Sections \ref{sectBasis}, \ref{sectGeneral}, and \ref{sectBar} are new, Theorems \ref{thCN}, \ref{thGeneral}, \ref{th5}, and \ref{th6} are presented for the first time.

\section{Generalized Vieta's formulas in geometric algebras}\label{sectVieta}

Let us consider the real (Clifford) geometric algebra $\cl_{p,q}$, $n=p+q\geq 1$ \cite{Hestenes,Doran,Lounesto} with the generators $e_a$, $a=1, 2, \ldots, n$ and the identity element $e\equiv 1$. The generators satisfy 
$$
e_a e_b+e_b e_a=2 \eta_{ab} e,\qquad  a, b =1, 2, \ldots, n,
$$
where $\eta=(\eta_{ab})$ is the diagonal matrix with its first $p$ entries equal to $1$ and the last $q$ entries equal to $-1$ on the diagonal. The grade involution and reversion of an arbitrary element (a multivector) $U\in\cl_{p,q}$ are denoted by 
$$
\widehat{U}=\sum_{k=0}^n(-1)^{k}\langle U \rangle_k,\qquad 
\widetilde{U}=\sum_{k=0}^n (-1)^{\frac{k(k-1)}{2}} \langle U \rangle_k,
$$
where $\langle U\rangle_k$ is the projection of $U$ onto the subspace $\cl_{p,q}^k$ of grade $k=0, 1, \ldots, n$. We have
\begin{eqnarray}
\widehat{UV}=\widehat{U} \widehat{V},\qquad \widetilde{UV}=\widetilde{V} \widetilde{U},\qquad \forall U, V\in\cl_{p,q}.\label{invol}
\end{eqnarray}

Let us consider the following faithful representation (isomorphism) of the complexified geometric algebra $\BC\otimes\cl_{p,q}$, $n=p+q$
\begin{eqnarray}
\beta: \BC\otimes\cl_{p,q} \to M_{p,q}:=
\left\lbrace
\begin{array}{ll}
\Mat(2^{\frac{n}{2}}, \BC) &\quad\mbox{if $n$ is even},\\
\Mat(2^{\frac{n-1}{2}}, \BC)\oplus\Mat(2^{\frac{n-1}{2}}, \BC)&\quad\mbox{if $n$ is odd}.
\end{array}
\right.
\end{eqnarray}
The real geometric algebra $\cl_{p,q}$ is isomorphic to some subalgebra of $M_{p,q}$, because $\cl_{p,q}\subset\BC\otimes\cl_{p,q}$ and we can consider the representation of not minimal dimension
$$\beta: \cl_{p,q} \to \beta(\cl_{p,q})\subset M_{p,q}.$$
We can introduce (see \cite{det}) the notion of determinant
$$\Det(U):=\det(\beta(U))\in\BR,\qquad U\in\cl_{p,q}$$
and the notion of characteristic polynomial
\begin{eqnarray}
&&\varphi_U(\lambda):=\Det(\lambda e-U)=\lambda^N-C_{(1)}\lambda^{N-1}-\cdots-C_{(N-1)}\lambda-C_{(N)}\in\cl^0_{p,q}\equiv\BR,\nonumber\\
&&U\in\cl_{p,q},\qquad N=2^{[\frac{n+1}{2}]},\qquad C_{(k)}=C_{(k)}(U)\in\cl^0_{p,q}\equiv\BR,\quad k=1, \ldots, N,\label{char}
\end{eqnarray}
where $\cl^0_{p,q}$ is a subspace of elements of grade $0$, which we identify with scalars. 

Let us denote the solutions of the characteristic equation $\varphi_U(\lambda)=0$ (i.e. eigenvalues) by $\lambda_1, \ldots, \lambda_N$. By the Vieta's formulas from matrix theory, we know that
$$C_{(k)}=(-1)^{k+1} \sum_{1\leq i_1 <i_2<\cdots<i_k\leq n} \lambda_{i_1}\lambda_{i_2}\cdots \lambda_{i_k},\qquad k=1, \ldots, N,$$
in particular,
$$C_{(1)}=\lambda_1+\cdots+\lambda_N={\rm Tr}(U),\qquad \ldots,\qquad C_{(N)}=-\lambda_1\cdots \lambda_N=-{\rm Det}(U),$$
where ${\rm Tr}(U):=\tr(\beta(U))=N \langle U \rangle_0$ is the trace of $U$. The elements $C_{(k)}$, $k=1, \ldots, N$ are elementary symmetrical polynomials in the variables $\lambda_1$, \ldots, $\lambda_N$.

\subsection{The case $n=1$}

Let us consider the particular case $n=1$. In this case, the geometric algebra $\cl_{p,q}$ is commutative and $N=2$. We have
\begin{eqnarray}
C_{(1)}=\lambda_1+\lambda_2 \in\BR,\qquad C_{(2)}=-\lambda_1 \lambda_2\in\BR.\label{n1lambda}
\end{eqnarray}
But also we have (see \cite{det})
\begin{eqnarray}
C_{(1)}=U+\widehat{U}\in\cl^0_{p,q}\equiv\BR,\qquad C_{(2)}=-U\widehat{U}\in\cl^0_{p,q}\equiv\BR.\label{n1}
\end{eqnarray}
The elements $y_1:=U$ and $y_2:=\widehat{U}$ are not scalars (and are not equal to the eigenvalues $\lambda_1$ and $\lambda_2$), but they are solutions of the characteristic equation $\varphi_U(x)=0$, $x=y_1, y_2$ by the Cayley--Hamilton theorem (see the details in Section \ref{cl23}). Using
$$\lambda^2-(U+\widehat{U})\lambda+U\widehat{U}=0,$$
we get the explicit formulas for the eigenvalues
\begin{eqnarray}
\lambda_{1,2}&=&\frac{1}{2}(U+\widehat{U}\pm \sqrt{(U+\widehat{U})^2-4U\widehat{U}})=\frac{1}{2}(U+\widehat{U}\pm \sqrt{(U-\widehat{U})^2})\nonumber\\
&=& \langle U\rangle_0 \pm \sqrt{(\langle U \rangle _1)^2},
\end{eqnarray}
which do not coincide with the explicit formulas for $y_{1,2}$
\begin{eqnarray}
y_{1,2}=\langle U\rangle_0 \pm \langle U\rangle_1,
\end{eqnarray}
because the scalar $\sqrt{(\langle U \rangle _1)^2}$ does not coincide with the vector (element of grade~$1$) $\langle U\rangle_1$. We see that the role of the roots $\lambda_{1,2}$ (which are complex scalars) of the characteristic equation is played by some combinations $y_{1,2}$ of involutions of elements (which are not scalars). In the case of degenerate eigenvalues, we have $\langle U \rangle_1=0$ and the coincidence $\lambda_{1,2}=y_{1,2}=U=\langle U\rangle_0$.

\subsection{The case $n=2$}

Let us consider the particular case $n=2$. We have $N=2$ and
\begin{eqnarray}
C_{(1)}=\lambda_1+\lambda_2 \in\BR,\qquad C_{(2)}=-\lambda_1 \lambda_2\in\BR.\label{n2lambda}
\end{eqnarray}
But also we have (see \cite{det})
\begin{eqnarray}
C_{(1)}=U+\widehat{\widetilde{U}}\in\cl^0_{p,q}\equiv\BR,\qquad C_{(2)}=-U\widehat{\widetilde{U}}\in\cl^0_{p,q}\equiv\BR.\label{n2}
\end{eqnarray}
Note that $U\widehat{\widetilde{U}}=\widehat{\widetilde{U}} U$ in the case $n=2$. The elements $y_1:=U$ and $y_2:=\widehat{\widetilde{U}}$ are not scalars (and are not equal to the eigenvalues $\lambda_1$ and $\lambda_2$), but they are solutions of the characteristic equation $\varphi_U(x)=0$, $x=y_1, y_2$ by the Cayley--Hamilton theorem (see the details in Section \ref{cl23}). Using
$$\lambda^2-(U+\widehat{\widetilde{U}})\lambda+U\widehat{\widetilde{U}}=0,$$
we get the explicit formulas for the eigenvalues
\begin{eqnarray}
\lambda_{1,2}&=&\frac{1}{2}(U+\widehat{\widetilde{U}}\pm \sqrt{(U+\widehat{\widetilde{U}})^2-4U\widehat{\widetilde{U}}})=\frac{1}{2}(U+\widehat{\widetilde{U}}\pm \sqrt{(U-\widehat{\widetilde{U}})^2})\nonumber\\
&=& \langle U\rangle_0 \pm \sqrt{(\langle U\rangle_1+\langle U\rangle_2)^2},
\end{eqnarray}
which do not coincide with the explicit formulas for $y_{1,2}$
\begin{eqnarray}
y_{1,2}=\langle U\rangle_0 \pm (\langle U\rangle_1+\langle U\rangle_2),
\end{eqnarray}
where the scalar $\sqrt{(\langle U\rangle_1+\langle U\rangle_2)^2}=\sqrt{(\langle U\rangle_1)^2+(\langle U\rangle_2)^2}$ is not equal to the expression $\langle U\rangle_1+\langle U\rangle_2$. The role of the roots $\lambda_{1,2}$ (which are complex scalars) of the characteristic equation is played by some combinations $y_{1,2}$ of involutions of elements (which are not scalars). 

In the case of degenerate eigenvalues, we have $\langle U \rangle_1=\langle U \rangle_2=0$ and the coincidence $\lambda_{1,2}=y_{1,2}=U=\langle U\rangle_0$ in the case of two Jordan blocks; or $(\langle U \rangle_1)^2=-(\langle U \rangle_2)^2\neq 0$ and $\lambda_{1,2}=\langle U \rangle_0\neq y_{1,2}=\langle U\rangle_0 \pm (\langle U\rangle_1+\langle U\rangle_2)$ in the case of one Jordan block. For example, for $U=5e+\frac{1}{2}(e_2+e_{12})$, we have $\lambda_{1,2}=5$ and $y_{1,2}=5e\pm\frac{1}{2}(e_2+e_{12})$ in the case $n=p=2$, $q=0$.

\subsection{The case $n=3$}

Let us consider the case $n=3$. We have $N=4$ and the formulas (see \cite{det})
\begin{eqnarray}
&&C_{(1)}=
    U+\widehat{U}+\widetilde{U}+\widehat{\widetilde{U}},\nonumber\\
    &&C_{(2)}=-(
    U\widetilde{U}+
    U\widehat{U}+
    U\widehat{\widetilde{U}}+
    \widehat{U}\widehat{\widetilde{U}}+
    \widetilde{U}\widehat{\widetilde{U}}+
    \widehat{U}\widetilde{U}),\nonumber\\
   &&C_{(3)}=
    U\widehat{U}\widetilde{U}+
    U\widehat{U}\widehat{\widetilde{U}}+
    U\widetilde{U}\widehat{\widetilde{U}}+
    \widehat{U}\widetilde{U}\widehat{\widetilde{U}},\nonumber\\
   &&C_{(4)}=-
    U\widehat{U}\widetilde{U}\widehat{\widetilde{U}}.\label{n3}
\end{eqnarray}
These formulas look like the ordinary Vieta's formulas for eigenvalues:
\begin{eqnarray}
&&C_{(1)}=
    \lambda_1+\lambda_2+\lambda_3+\lambda_4,\nonumber\\
&&C_{(2)}=-(
    \lambda_1\lambda_2+
    \lambda_1\lambda_3+
    \lambda_1\lambda_4+
    \lambda_2\lambda_3+
    \lambda_2\lambda_4+
    \lambda_3\lambda_4),\nonumber\\
   && C_{(3)}=
    \lambda_1\lambda_2\lambda_3+
    \lambda_1\lambda_2\lambda_4+
    \lambda_1\lambda_3\lambda_4+
    \lambda_2\lambda_3\lambda_4,\nonumber\\
    &&C_{(4)}=-
    \lambda_1\lambda_2\lambda_3\lambda_4.\label{n3lambda}
\end{eqnarray}
The elements 
\begin{eqnarray*}
&&\!\!y_1:=U=\langle  U \rangle_0+\langle U\rangle_1+\langle U \rangle_2+ \langle U \rangle_3,\quad y_2:=\widetilde{U}=\langle  U \rangle_0+\langle U\rangle_1-\langle U \rangle_2- \langle U \rangle_3,\\
&&\!\!y_3:=\widehat{U}=\langle  U \rangle_0-\langle U\rangle_1+\langle U \rangle_2- \langle U \rangle_3,\quad y_4:=\widetilde{\widehat{U}}=\langle  U \rangle_0-\langle U\rangle_1-\langle U \rangle_2+ \langle U \rangle_3,
\end{eqnarray*}
 are not scalars (and are not equal to the eigenvalues $\lambda_1, \lambda_2, \lambda_3, \lambda_4$), but they are solutions of the characteristic equation $\varphi_U(x)=0$, $x=y_1, y_2, y_3, y_4$ by the Cayley--Hamilton theorem (see the details in Section \ref{cl23}). 

\bigskip

We call the formulas (\ref{n1}), (\ref{n2}), (\ref{n3}) and their analogues for the cases $n\geq 4$ {\it generalized Vieta's formulas in geometric algebra}. The formulas (\ref{n1}), (\ref{n2}), (\ref{n3}) were proved in \cite{det} using recursive formulas for the characteristic polynomial coefficients following from the Faddeev--LeVerrier algorithm. In this paper, we present an alternative proof of these formulas using the techniques of noncommutative symmetric functions (see Sections \ref{secn23} and \ref{cl23}).

\section{On Gelfand--Retakh noncommutative Vieta theorem}\label{secn23}

Let us discuss the following Gelfand--Retakh theorem (known as the {\it noncommutative Vieta theorem} \cite{Gelfand}, see also \cite{Sch1,Sch2}). In Section \ref{cl23}, we use it for the characteristic polynomial in geometric algebras.

\begin{theorem}[\cite{Gelfand}]\label{thGel} If $\{x_1, \ldots, x_N\}$ is an ordered generic set (i.e. Vandermonde quasideterminants $v_k$ are defined and invertible for all $k=1, \ldots, N$) of solutions of the equation
\begin{eqnarray}
P_N(x):=x^N-a_1x^{N-1}-\cdots-a_N=0\label{PN}
\end{eqnarray}
over a skew-field, then for $k=1, 2, \ldots, N$:
$$a_k=(-1)^{k+1} \sum_{1\leq i_1 <i_2<\cdots<i_k\leq N} y_{i_k}\cdots y_{i_1},$$
where
$$y_k=v_k x_k v_k^{-1}.$$
\end{theorem}
In \cite{Gelfand}, the definition of Vandermonde quasideterminants $v_k$ is given (see also \cite{Gelfand2,Gelfand3}). In this paper, we use another definition of the elements $v_k$ from \cite{Fung}:
\begin{eqnarray}
\!\!\!v_k=P_{k-1}(x_{k})=x_k^{k-1}-(y_{k-1}+\cdots+y_1)x_k^{k-2}+\cdots+ (-1)^{k-1}y_{k-1}\cdots y_1.
\end{eqnarray}
In particular, we have
\begin{eqnarray}
v_1=1,\qquad v_2=x_2-y_1,\qquad v_3=x_3^2-(y_2+y_1)x_3+y_2y_1.
\end{eqnarray}
Let us give examples.

In the case $N=1$, substituting $x=x_1$ into (\ref{PN}), we obtain $y_1=x_1=a_1$ and $v_1=1$.

In the case $N=2$, the equation (\ref{PN}) with $a_1=y_2+y_1$ and $a_2=-y_2 y_1$ can be rewritten in the form
\begin{eqnarray}
(x-y_1)x-y_2(x-y_1)=0.\label{r0}
\end{eqnarray}
Substituting $x=x_1$ into (\ref{r0}), we obtain $(x_1-y_1)x_1-y_2(x_1-y_1)=0$ and we can take $y_1=x_1$.
Substituting $x=x_2$ into (\ref{r0}), we obtain
\begin{eqnarray}
(x_2-y_1)x_2-y_2(x_2-y_1)=0\label{r1}
\end{eqnarray}
and we can take $y_2=v_2 x_2 v_2^{-1}$ in the case of invertible $v_2:=x_2-y_1=x_2-x_1$.
If $[x_2, v_2]=0$, then $y_2=x_2$ by (\ref{r1}).

In the case $N=3$, the equation (\ref{PN}) with $a_1=y_3+y_2+y_1$, $a_2=-(y_3 y_2+y_3y_1+y_2y_1)$, and $a_3=y_3y_2y_1$ can be rewritten in the form
\begin{eqnarray}
(x^2-(y_2+y_1)x+y_2 y_1)x-y_3(x^2-(y_2+y_1)x+y_2 y_1)=0.\label{r2}
\end{eqnarray}
Substituting $x=x_1$ and $x=x_2$ into (\ref{r2}), we conclude that we can take again $y_1=x_1$ and $y_2=v_2 x_2 v_2^{-1}$ in the case of invertible $v_2=x_2-y_1=x_2-x_1$. If $[x_2, v_2]=0$, then $y_2=x_2$. Substituting $x=x_3$ into (\ref{r2}), we obtain
$(x_3^2-(y_2+y_1)x_3+y_2 y_1)x_3-y_3(x_3^2-(y_2+y_1)x_3+y_2 y_1)=0$ and we can take $y_3=v_3 x_3 v_3^{-1}$ in the case of invertible $v_3:=x_3^2-(y_2+y_1)x_3+y_2 y_1$. We get $y_3=x_3$ in the case $[x_3, v_3]=0$. And so on.

\begin{remark}
The condition $[v_k, x_k]=0$ is equivalent to
$$[E_j, x_k]=0,\qquad j=1, \ldots, k-1,$$
where $E_j$, $j=1, \ldots, k-1$ are noncommutative elementary symmetric polynomials in the variables $y_{k-1}$, \ldots, $y_1$:
$$E_1=y_{k-1}+\cdots+y_1,\qquad E_{k-1}=y_{k-1}\cdots y_2 y_1.$$
\end{remark}
For example, in the particular case $N=4$, when all $[v_k, x_k]=0$, $k=1, \ldots, N$, we can take $y_k=x_k$, $k=1, 2, 3, 4$, in the case
\begin{eqnarray*}
&&[x_2, x_1]=0,\quad [x_3, x_2x_1]=0,\quad [x_3, x_2+x_1]=0,\quad [x_4, x_3x_2x_1]=0,\\
&&[x_4, x_3x_2+x_3x_1+x_2x_1]=0,\quad [x_4, x_3+x_2+x_1]=0.
\end{eqnarray*}
We use this particular case below in $\cl_{p,q}$ with $n=p+q=3$.

\section{Application of noncommutative Vieta theorem to geometric algebras}\label{cl23}

Let us apply Theorem \ref{thGel} to the particular case of the characteristic polynomial $\varphi_U(\lambda)$ in geometric algebra $\cl_{p,q}$. The elements $a_k=C_{(k)}\in\BR$, $k=1, \ldots, N$ from (\ref{PN}) are scalars now. We need $N$ solutions $x_1, x_2, \ldots x_N$ of the characterstic equation $\varphi_U(x)=0$. By the Cayley--Hamilton theorem, we can take $x_1=U$: 
\begin{eqnarray}
\varphi_U (U)=0.
\end{eqnarray}
We have the following statement.

\begin{theorem}\label{thCH} Let us consider the characteristic polynomial $\varphi_U(\lambda)$ (\ref{char}) of an arbitrary element $U\in\cl_{p,q}$. We have
\begin{eqnarray}
&&\varphi_U(\lambda)=\varphi_{\widehat{U}}(\lambda)=\varphi_{\widetilde{U}}(\lambda)=\varphi_{\widetilde{\widehat{U}}}(\lambda),\\
&&\varphi_U(\widetilde{U})=\varphi_U(\widehat{U})=\varphi_U(\widehat{\widetilde{U}})=0.
\end{eqnarray}
\end{theorem}
\begin{proof} We know that (see Lemma 10 in \cite{det})
$$\Det(U)=\Det(\widehat{U})=\Det(\widetilde{U})=\Det(\widehat{\widetilde{U}}).$$
Using the definition of characteristic polynomial (\ref{char}), we get
$$\varphi_{\widehat{U}}(\lambda)=\Det(\lambda e - \widehat{U})=\Det(\widehat{\lambda e - U})=\Det(\lambda e-U)=\varphi_U(\lambda).
$$
Using the Cayley--Hamilton theorem $\varphi_U(U)=0$, we get $\varphi_{\widehat{U}}(U)=0$. Substituting $\widehat{U}$ for $U$, we get $\varphi_U(\widehat{U})=0$. We obtain the other formulas in a similar way. \qed
\end{proof}

\subsection{The case $n=1$}

In this case, the geometric algebra is commutative. We can take $y_1=x_1=U$ in Theorem \ref{thGel} by the Cayley--Hamilton theorem. The element $x_2=\widehat{U}$ satisfies the characteristic equation (see Theorem~\ref{thCH}). We have $v_2=x_2-x_1=\widehat{U}-U=-2\langle U \rangle_1$. If $\langle U \rangle_1\neq 0$, then $y_2=x_2=\widehat{U}$ and we obtain the formulas (\ref{n1}).

\subsection{The case $n=2$}

We can take $y_1=x_1=U$ in Theorem \ref{thGel} by the Cayley--Hamilton theorem. The element $x_2=\widetilde{U}$ satisfies the characteristic equation (see Theorem~\ref{thCH}). We have $v_2=x_2-x_1=\widetilde{U}-U=-2\langle U \rangle_2$. If $\langle U \rangle_2=0$, then we can use the formulas from the case $n=1$. If $\langle U \rangle_2\neq 0$, then $v_2=\lambda e_{12}$, $\lambda\neq 0 $ is invertible and
$y_2=v_2 \widetilde{U} v_2^{-1}= \widetilde{\widehat{U}}$ because the pseudoscalar $e_{12}$ commutes with all even elements and anticommutes with all odd elements. We get the formulas (\ref{n2}).

\subsection{The case $n=3$}

We can take $y_1=x_1=U$ in Theorem \ref{thGel} by the Cayley--Hamilton theorem. 

Let us consider $x_2=\widehat{\widetilde{U}}$. We have $v_2=x_2-x_1=\widehat{\widetilde{U}}-U$ and $[v_2, x_2]=0$ because $[U,  \widehat{\widetilde{U}}]=0$ in the case $n=3$. Thus we can take $y_2=x_2$.

Let us consider $x_3=\widehat{U}$. We have
$$[x_3, v_3]=[x_3, x_3^2-(x_1+x_2)x_3+x_2x_1]=0,$$
because the elements $x_1+x_2=U+\widehat{\widetilde{U}}$ and $x_2 x_1=\widehat{\widetilde{U}} U$ belong to the center ${\rm Cen}(\cl_{p,q})=\cl^0_{p,q}\oplus\cl^3_{p,q}$, and can take $y_3=x_3$.

Let us consider $x_4=\widetilde{U}$. We have
$$[x_4, v_4]=[x_4, x_4^3-(x_3+x_2+x_1)x_4^2+(x_3x_2+x_3x_1+x_2x_1)x_4-(x_3x_2x_1)]=0,$$
because the elements $x_1+x_2=U+\widehat{\widetilde{U}}$ and $x_2 x_1=\widehat{\widetilde{U}} U$ belong to the center ${\rm Cen}(\cl_{p,q})$, and $x_3 x_4=x_4 x_3$, i.e. $\widehat{U} \widetilde{U}=\widetilde{U}\widehat{U}$ in the case $n=3$. We take $y_4=x_4$.

We obtain $y_k=x_k$, $k=1, 2, 3, 4$ and the following formulas, which are another version of the formulas (\ref{n3}):
\begin{eqnarray}
&&C_{(1)}=
    \widetilde{U}+\widehat{U}+\widehat{\widetilde{U}}+U,\nonumber\\
    &&C_{(2)}=-(
    \widetilde{U} \widehat{U}+\widetilde{U}\widehat{\widetilde{U}}+\widetilde{U} U+ \widehat{U}\widehat{\widetilde{U}}+ \widehat{U}U+\widehat{\widetilde{U}}U),\nonumber\\
   &&C_{(3)}=
   \widetilde{U} \widehat{U}\widehat{\widetilde{U}}+\widetilde{U} \widehat{U}U+\widetilde{U} \widehat{\widetilde{U}}U+\widehat{U}\widehat{\widetilde{U}}U,\nonumber\\
   &&C_{(4)}=-
\widetilde{U} \widehat{U}\widehat{\widetilde{U}}U.\label{n3an}
\end{eqnarray}
Note that we obtain these formulas for the element $U$ with invertible expressions $v_2$, $v_3$, and $v_4$ (for other elements $U$, other sequences $x_1$, $x_2$, $x_3$, $x_4$ can be considered). Also note that not every sequence $y_1$, $y_2$, $y_3$, $y_4$ from $\{\widetilde{U}, \widehat{U}, \widetilde{\widehat{U}}, U\}$ gives the correct Vieta's formulas (see Theorem 3 and Lemma 7 in \cite{det}, the formulas (\ref{n3}) and (\ref{n3an}) are two of several correct forms).

\subsection{The cases $n\geq 4$}\label{secn4}

The generalized Vieta's formulas in the cases $n\geq 4$ are more complicated. We use the additional (triangle) operation (see \cite{det})
\begin{eqnarray}
U^\bigtriangleup :=\sum_{k=0}^n (-1)^{\frac{k(k-1)(k-2)(k-3)}{24}}\langle U \rangle_k=\!\!\!\!\!\sum_{k=0, 1, 2, 3\!\!\mod 8}\!\!\!\!\!\langle U \rangle_k - \!\!\!\!\!\sum_{k=4, 5, 6, 7\!\!\mod 8}\!\!\!\!\!\langle U \rangle_k.
\end{eqnarray}
Note that
\begin{eqnarray}
\Det(U^\bigtriangleup )\neq \Det(U),\qquad \varphi_{U^\bigtriangleup }(\lambda)\neq\varphi_U(\lambda),\qquad\varphi_U(U^\bigtriangleup)\neq 0
\end{eqnarray}
in the general case (compare with the statements of Theorem~\ref{thCH}).

In the case $n=4$, the generalized Vieta's formulas have the following form 
\begin{eqnarray}
 &&C_{(1)}=
    U+\widehat{\widetilde{U}}+\widehat{U}^\bigtriangleup+\widetilde{U}^\bigtriangleup,\nonumber\\
    &&C_{(2)}=-(
    U\widehat{\widetilde{U}}+
    U\widehat{U}^\bigtriangleup+
    U\widetilde{U}^\bigtriangleup+
    \widehat{\widetilde{U}}\widehat{U}^\bigtriangleup+
    \widehat{\widetilde{U}}\widetilde{U}^\bigtriangleup+
    (\widehat{U}\widetilde{U})^\bigtriangleup),\nonumber\\
    &&C_{(3)}=
    U\widehat{\widetilde{U}}\widehat{U}^\bigtriangleup+
    U\widehat{\widetilde{U}}\widetilde{U}^\bigtriangleup+
    U(\widehat{U}\widetilde{U})^\bigtriangleup+
    \widehat{\widetilde{U}}(\widehat{U}\widetilde{U})^\bigtriangleup,\nonumber\\
    &&C_{(4)}=-
    U\widehat{\widetilde{U}}(\widehat{U}\widetilde{U})^\bigtriangleup,\label{n41}
\end{eqnarray}
where the coefficients $C_{(k)}$, $k=1, 2, 3, 4$ are not elementary symmetrical polynomials because of the additional operation of conjugation $\bigtriangleup$. These formulas look like the ordinary Vieta's formulas
\begin{eqnarray}
&&C_{(1)}=
    \lambda_1+\lambda_2+\lambda_3+\lambda_4,\nonumber\\
&&C_{(2)}=-(
    \lambda_1\lambda_2+
    \lambda_1\lambda_3+
    \lambda_1\lambda_4+
    \lambda_2\lambda_3+
    \lambda_2\lambda_4+
    \lambda_3\lambda_4),\nonumber\\
   && C_{(3)}=
    \lambda_1\lambda_2\lambda_3+
    \lambda_1\lambda_2\lambda_4+
    \lambda_1\lambda_3\lambda_4+
    \lambda_2\lambda_3\lambda_4,\nonumber\\
    &&C_{(4)}=-
    \lambda_1\lambda_2\lambda_3\lambda_4,\label{n41V}
\end{eqnarray}
if we ignore the operation $\bigtriangleup$. Note that
$$
(UV)^\bigtriangleup\neq U^\bigtriangleup V^\bigtriangleup, \qquad (UV)^\bigtriangleup\neq V^\bigtriangleup U^\bigtriangleup
$$
in the general case (compare with the properties (\ref{invol})).

The  analogues of the formulas (\ref{n41}) for the cases $n=5, 6$ are presented in \cite{Abd2} (see Theorem 5.1 and Section 8). These formulas also have the form of elementary symmetric polynomials, only if
we ignore the operation $\bigtriangleup$, and can be interpreted as generalized noncommutative Vieta's formulas. These formulas do not follow directly from the Gelfand--Retakh noncommutative Vieta theorem, it is not easy task to guess the ``right'' (generic) ordered set of solutions $x_1$, $x_2$, $x_3$, \ldots $x_N$ of the characteristic equation to obtain the elements $y_1$, $y_2$, $y_3$, \ldots, $y_N$ we need in the generalized Vieta's formulas. This is a task for further research.

We discuss Vieta's formulas in more complicated cases $n\geq 4$ in details using different techniques in Sections \ref{sectGeneral} and \ref{sectBar}.

\section{On basis-free formulas for the determinant}\label{sectBasis}

In \cite{det}, we present a method based on the Faddeev--LeVerrier algorithm that allows us to recursively obtain basis-free formulas for all the characteristic coefficients $C_{(k)}$, $k=1, \ldots, N$ (\ref{char}):
\begin{eqnarray}
\!\!\!\!\!U_{(1)}:=U,\quad U_{(k+1)}=U(U_{(k)}-C_{(k)}),\quad C_{(k)}:=\frac{N}{k}\langle U_{(k)} \rangle_0,\quad k=1, \ldots, N.
\label{FL}\end{eqnarray}
We use this method also in \cite{Abd2,Abd}. In this method, we obtain high coefficients from the lowest ones. The determinant is minus the last coefficient
\begin{eqnarray}
\Det(U)=-C_{(N)}=-U_{(N)}=U(C_{(N-1)}-U_{(N-1)}).\label{laststep}
\end{eqnarray}
Let us consider the operations $\bigtriangleup_1$, \ldots $\bigtriangleup_m$, where $m=[\log_2(n)]+1$ (see \cite{det}):
\begin{eqnarray}
U^{\bigtriangleup_j}:=\sum_{k=0}^n (-1)^{C_k^{2^{j-1}}}\langle U \rangle_k,\qquad j=1, \ldots, m,\qquad C_k^i=\frac{k!}{i!(k-i)!}.\label{opconj}
\end{eqnarray}
In particular cases, we have $U^{\bigtriangleup_1}=\widehat{U}$, $U^{\bigtriangleup_2}=\widetilde{U}$, $U^{\bigtriangleup_3}=U^\bigtriangleup$.

\begin{theorem}\label{thCN} For arbitrary $n\geq 1$, there exists a basis-free formula for $\Det(U)=-C_{(N)}$ of the following form
\begin{eqnarray}
\Det(U)=U (\sum_{j=1}^t \alpha_j H_j(U))=(\sum_{j=1}^t \alpha_j H_j(U)) U,\quad \alpha_j\in{\mathbb R},\quad \sum_{j=1}^t \alpha_j=1,\label{type}
\end{eqnarray}
where each expression $H_j(U)$ involves only multiplication (consists of $N-1$ multipliers) and the operations $\bigtriangleup_1$, \ldots, $\bigtriangleup_m$, $m=[\log_2(n)]+1$.
\end{theorem}
Note that the expression $$\Adj(U):=\sum_{j=1}^t \alpha_j H_j(U)$$ in (\ref{type}) is the {\it adjugate} of $U$ and it is used to calculate the inverse $$U^{-1}=\frac{\Adj(U)}{\Det(U)}$$ in the case $\Det(U)\neq 0$.

\begin{proof}
The determinant can be calculated using the Faddeev--LeVerrier algorithm (\ref{FL}), so it involves only multiplication, summation, and the operation $\langle\quad \rangle_0$ of projection onto the subspace of grade $0$. The operation $\langle\quad\rangle_0$ can be realized using $m$ operations $\bigtriangleup_1$, \ldots $\bigtriangleup_m$ (see Theorem 1 in \cite{det}):
$$
\langle U \rangle_0=\frac{1}{2^m}(U+U^{\bigtriangleup_1}+U^{\bigtriangleup_2}+\cdots+U^{\bigtriangleup_1\ldots \bigtriangleup_m}).
$$
The common multiplier $U$ is in (\ref{type}) because of the last step (\ref{laststep}) of the Faddeev--LeVerrier algorithm. Each term $H_j(U)$ consists of $N-1$ multipliers because of the property (see \cite{det}) 
$$\Det(\lambda U)=\lambda^N \Det(U).$$
We have $\sum_{j=1}^t \alpha_j=1$, because we have $$\Det(e)=1$$ for the identity element $e$ by the definition (we associate the identity matrix with the identity element, see \cite{det}) and all the operations (\ref{opconj}) do not change the sign of grade $0$: $H_j(e)=1$, $j=1, \ldots, t$. \qed
\end{proof}

\begin{example} In the cases $n\leq 4$, we have the formulas (\ref{n1}), (\ref{n2}), (\ref{n3}), (\ref{n41}) of type (\ref{type}) with $t=1$. In the case $n=5$, we have
\begin{eqnarray}
\Det(U)=U\widehat{\widetilde{U}}\widehat{U}\widetilde{U}(\widehat{U}\widetilde{U}U\widehat{\widetilde{U}})^\bigtriangleup,\qquad t=1.\label{n5}
\end{eqnarray}
In the case $n=6$, we have
\begin{eqnarray}
\!\!\!\!\!\!\!\Det(U)=\frac{1}{3}U\widetilde{U}\widehat{U}\widehat{\widetilde{U}}(\widehat{U}\widehat{\widetilde{U}} U \widetilde{U})^\bigtriangleup+\frac{2}{3}U\widetilde{U}((\widehat{U}\widehat{\widetilde{U}})^\bigtriangleup((\widehat{U}\widehat{\widetilde{U}})^\bigtriangleup (U \widetilde{U})^\bigtriangleup)^\bigtriangleup)^\bigtriangleup,\,\,\,\, t=2.\label{n6}
\end{eqnarray}
It is proved that the formula of type (\ref{type}) with $t=1$ (see Theorem \ref{thCN}) does not exist in the case $n=6$ using computer calculations in \cite{acus}.
\end{example}

We call the formula (\ref{type}) for $\Det(U)$  (see Theorem \ref{thCN}) {\it simple}, if the parameter $t$ is minimal. For the cases $n\geq 7$, the problem of finding an explicit formula for $\Det(U)$ of type (\ref{type}) with minimal $t$ (as well as finding the value of minimal $t$) is open.

\section{Generalized Vieta theorem for arbitrary $n$}\label{sectGeneral}

In this section, we present a method to obtain basis-free formulas for all characteristic polynomial coefficients, if we know the basis-free formula for the last characteristic polynomial coefficient $C_{(N)}=-\Det(U)$. In this method, we get low coefficients from the highest one (in contrast to the method (\ref{FL}) based on the Faddeev--LeVerrier algorithm, where we get high coefficients from low ones). We call Theorem \ref{thGeneral} {\it generalized Vieta theorem in Clifford geometric algebras}. 

\begin{example} Let us start with an example $n=4$. Let us look at the formula for the last characteristic polynomial coefficient $\Det(U)=-C_{(4)}$ (\ref{n41}):
$$\Det(U)=-C_{(4)}=
    U\widehat{\widetilde{U}}(\widehat{U}\widetilde{U})^\bigtriangleup$$
and formally change the expression $(\widehat{U} \widetilde{U})^\bigtriangleup$ to $\widehat{U}^\bigtriangleup \widetilde{U}^\bigtriangleup$ (which is not the same). Let us write down symmetric polynomials for the resulting expression:
\begin{eqnarray}
 &&U+\widehat{\widetilde{U}}+\widehat{U}^\bigtriangleup+\widetilde{U}^\bigtriangleup,\qquad 
    U\widehat{\widetilde{U}}+
    U\widehat{U}^\bigtriangleup+
    U\widetilde{U}^\bigtriangleup+
    \widehat{\widetilde{U}}\widehat{U}^\bigtriangleup+
    \widehat{\widetilde{U}}\widetilde{U}^\bigtriangleup+
    \widehat{U}^\bigtriangleup\widetilde{U}^\bigtriangleup,\nonumber\\
    &&
    U\widehat{\widetilde{U}}\widehat{U}^\bigtriangleup+
    U\widehat{\widetilde{U}}\widetilde{U}^\bigtriangleup+
    U\widehat{U}^\bigtriangleup\widetilde{U}^\bigtriangleup+
    \widehat{\widetilde{U}}\widehat{U}^\bigtriangleup\widetilde{U}^\bigtriangleup,\qquad 
    U\widehat{\widetilde{U}}\widehat{U}^\bigtriangleup\widetilde{U}^\bigtriangleup,
\end{eqnarray}
Let us formally change back $\widehat{U}^\bigtriangleup \widetilde{U}^\bigtriangleup$ to $(\widehat{U} \widetilde{U})^\bigtriangleup$ in these formulas. We obtain the right formulas (\ref{n41}) for $C_{(1)}$, $-C_{(2)}$, $C_{(3)}$, and $-C_{(4)}$.

The proof for this example is rather simple. We have
\begin{eqnarray}
\varphi_U(\lambda)&=&\Det(\lambda e-U)=(\lambda e-U)(\lambda e-U)\widehat{\widetilde{\quad}}((\lambda e-U)\widehat{\quad}(\lambda e-U)\widetilde{\quad})^\bigtriangleup\label{proof1}\\
&=&(\lambda e-U)(\lambda e-\widehat{\widetilde{U}})((\lambda e-\widehat{U})(\lambda e-\widetilde{U}))^\bigtriangleup\\
&=&(\lambda e-U)(\lambda e-\widehat{\widetilde{U}})(\lambda^2 e - \lambda \widehat{U} - \lambda \widetilde{U}+\widehat{U} \widetilde{U})^\bigtriangleup\\
&=&(\lambda e-U)(\lambda e-\widehat{\widetilde{U}})(\lambda^2 e - \lambda \widehat{U}^\bigtriangleup - \lambda \widetilde{U}^\bigtriangleup+(\widehat{U} \widetilde{U})^\bigtriangleup)\label{proof3}\\
&=&\lambda^4-(U+\widehat{\widetilde{U}}+\widehat{U}^\bigtriangleup+\widetilde{U}^\bigtriangleup)\lambda^3+(
    U\widehat{\widetilde{U}}+
    U\widehat{U}^\bigtriangleup+
    U\widetilde{U}^\bigtriangleup+
    \widehat{\widetilde{U}}\widehat{U}^\bigtriangleup\nonumber\\
    &+&
    \widehat{\widetilde{U}}\widetilde{U}^\bigtriangleup+
    (\widehat{U}\widetilde{U})^\bigtriangleup)\lambda^2-( U\widehat{\widetilde{U}}\widehat{U}^\bigtriangleup+
    U\widehat{\widetilde{U}}\widetilde{U}^\bigtriangleup+
    U(\widehat{U}\widetilde{U})^\bigtriangleup\nonumber\\
    &+&
    \widehat{\widetilde{U}}(\widehat{U}\widetilde{U})^\bigtriangleup)\lambda+U\widehat{\widetilde{U}}(\widehat{U}\widetilde{U})^\bigtriangleup.\label{proof4}
\end{eqnarray}
Note that if we formally change $(\widehat{U} \widetilde{U})^\bigtriangleup$ to $\widehat{U}^\bigtriangleup \widetilde{U}^\bigtriangleup$ (which is not the same) in (\ref{proof3}), then the whole expression will be
$(\lambda e-U)(\lambda e-\widehat{\widetilde{U}})(\lambda e-\widehat{U}^\bigtriangleup)(\lambda e-\widetilde{U}^\bigtriangleup)$ and the expression (\ref{proof4}) will have coefficients with elementary symmetric polynomials. We can not do that because $(\widehat{U} \widetilde{U})^\bigtriangleup\neq\widehat{U}^\bigtriangleup \widetilde{U}^\bigtriangleup$, so in (\ref{proof4}) we have $(\widehat{U} \widetilde{U})^\bigtriangleup$ instead of $\widehat{U}^\bigtriangleup \widetilde{U}^\bigtriangleup$ and do not have elementary symmetric polynomials.
\end{example} 

\begin{note} Note that we can also use the standard method of computing the coefficients for any polynomial. All $C_{(k)}$ coefficients for $k < N$ can be computed recursively by differentiating the determinant ${\rm Det}(\lambda e - U)$  with respect to $\lambda$ and equating $\lambda$ to zero:
    \begin{eqnarray*}
        &&D_{(N)} (\lambda) := -\varphi_U(\lambda)=-{\rm Det}(\lambda e -U), \quad U\in\cl_{p,q}, \quad\lambda\in\mathbb{R},\\
        &&D_{(k-1)}(\lambda):=\frac{1}{N-(k-1)}\frac{ \partial D_{(k)}(\lambda)}{ \partial\lambda},
        \quad  k=N,\ldots,1,\\
        &&C_{(k)} = D_{(k)}(0).
    \end{eqnarray*}
We can calculate the coefficients directly using the formulas
 $$D_{(k)}(\lambda)=\frac{1}{(N-k)!}\frac{ \partial^{N-k} D_{(N)}(\lambda)}{ \partial\lambda^{N-k}}.$$
This method gives the same result as the method discussed above. This method is implicit (and, of course, is good in computer calculations).
The previous method is explicit and closer to Vieta theorem in the sense that looking at the formula (\ref{n41}) for $C_{(4)}$ and removing unnecessary multipliers, we immediately get the formulas for $C_{(3)}$, $C_{(2)}$ and $C_{(1)}$, as we do in the ordinary Vieta theorem (\ref{n41V}), where we deal with elementary symmetric polynomials.
\end{note}

Now let us consider the case of arbitrary $n$.

\begin{theorem}\label{thGeneral}
Let us consider an arbitrary element $U\in\cl_{p,q}$, $n=p+q$, $N=2^{[\frac{n+1}{2}]}$ with the explicit formula for $\Det(U)=-C_{(N)}$ of the type (\ref{type}). Let us formally change all expressions $(U_1 \cdots U_j)^{\bigtriangleup_i}$ for $i=3, \ldots, m$ to $U_1^{\bigtriangleup_i} \cdots U_j^{\bigtriangleup_i}$ in each $H_j(U)$ separately. Let us write down elementary symmetric polynomials for these expressions separately. Let us formally change back previously considered expressions $U_1^{\bigtriangleup_i} \cdots U_j^{\bigtriangleup_i}$ for $i=3, \ldots, m$ to $(U_1 \cdots U_j)^{\bigtriangleup_i}$ in these formulas. We obtain the right formulas for the coefficients $(-1)^{k+1} C_{(k)}$, $k=1, \ldots, N-1$, which we call the generalized Vieta's formulas in Clifford geometric algebras.
\end{theorem}

\begin{proof} The proof is analogous to the example $n=4$ above (see (\ref{proof1}) -- (\ref{proof4})). We use the existing formula of type (\ref{type}) for the characteristic polynomial $\varphi_U(\lambda)=\Det(\lambda e-U)$. Then we drag the operations $\widehat{\quad}$, $\widetilde{\quad}$ inside the brackets $(\lambda e-U)$. These operations are linear and do not change the identity element $e$. We expand the brackets inside expressions with the operations $\bigtriangleup_3$, \ldots, $\bigtriangleup_m$. Then we drag the operations $\bigtriangleup_3$, \ldots, $\bigtriangleup_m$ inside the brackets and use linearity of these operations and that they do not change the identity element $e$. Then we expand all the brackets and get what we need. \qed
\end{proof}

\begin{note}\label{note2} We can reformulate Theorem \ref{thGeneral} in the following way. Let us consider an arbitrary element $U\in\cl_{p,q}$, $n=p+q$, $N=2^{[\frac{n+1}{2}]}$ with the explicit formula for $C_{(N)}=-\Det(U)$ of type (\ref{type}). Then all other characteristic polynomial coefficients are
    \begin{eqnarray}
        C_{(k)}=(-1)^{k+1}\sum_{X_l\in X(k)}{\rm F}(X_l),\quad
        k = 1,2,\ldots,N,\label{cross<5}
    \end{eqnarray}
    where ${\rm F}$ is $N$-variable function obtained from (\ref{type}) by enumerating every occurrence of the element $U$ from left to right (in each of $t$ terms respectively), and $X(k)$ is the set of all possible tuples with $k$ elements $U$ and $N-k$ identity elements $e$:
    \begin{eqnarray}
        \!\!\!\!\!X(k)=\left\lbrace X_l=(x_1,x_2,\ldots,x_N) \, | \,
                    x_i\in\lbrace e,U\rbrace, \quad \sum_{i=1}^N x_i=kU+(N-k)e
                \right\rbrace.\label{Xk}
    \end{eqnarray}
This statement was verified using computer calculations for the cases $n\leq 6$ in \cite{Abd2}. Now we get this statement for the case of arbitrary dimension $n$, because it is just reformulation of Theorem \ref{thGeneral} in other terms. Substituting $e$ into (\ref{Xk}) is the same as removing  unnecessary multipliers because the operations $\bigtriangleup_1$, \ldots, $\bigtriangleup_m$ do not change the identity element $e$.

In the cases of small dimensions, we have
  \begin{eqnarray}
        n=1,&&
        {\rm F}(x_1,x_2)=x_1\widehat{x_2}\in\cl_{p,q},
        \quad x_1,x_2\in\cl_{p,q};\nonumber\\
        n=2,&&
        {\rm F}(x_1,x_2)=x_1\widehat{\widetilde{x_2}}\in\cl_{p,q},
        \quad x_1,x_2\in\cl_{p,q};\nonumber\\
        n=3,&&
        {\rm F}(x_1,x_2,x_3,x_4)=x_1\widehat{x_2}\widetilde{x_3}\widehat{\widetilde{x_4}}\in\cl_{p,q},
        \quad x_1,x_2,x_3,x_4\in\cl_{p,q};\nonumber\\
        n=4,&&
        {\rm F}(x_1,x_2,x_3,x_4)=x_1\widehat{\widetilde{x_2}}(\widehat{x_3}\widetilde{x_4})^\bigtriangleup\in\cl_{p,q},
        \quad x_1,x_2,x_3,x_4\in\cl_{p,q};\nonumber\\
        n=5,&&
        {\rm F}(x_1,x_2,\ldots,x_8)=
        x_1\widehat{\widetilde{x_2}}\widehat{x_3}\widetilde{x_4}(\widehat{x_5}\widetilde{x_6}x_7\widehat{\widetilde{x_8}})^\bigtriangleup\in\cl_{p,q},\nonumber\\
        && x_1,x_2,\ldots,x_8\in\cl_{p,q};\nonumber\\
        n=6,&&
        {\rm F}(x_1,x_2,\ldots,x_8)=
        \frac{1}{3}x_1\widetilde{x_2}\widehat{x_3}\widehat{\widetilde{x_4}}(\widehat{x_5}\widehat{\widetilde{x_6}} x_7 \widetilde{x_8})^\bigtriangleup\nonumber\\
        &&+\frac{2}{3}x_1\widetilde{x_2}((\widehat{x_3}\widehat{\widetilde{x_4}})^\bigtriangleup((\widehat{x_5}\widehat{\widetilde{x_6}})^\bigtriangleup (x_7 \widetilde{x_8})^\bigtriangleup)^\bigtriangleup)^\bigtriangleup,\qquad x_1,x_2,\ldots,x_8\in\cl_{p,q}.\nonumber
    \end{eqnarray}
\end{note}

\begin{example}
In the cases $n=5, 6$, we can use the procedure from Theorem \ref{thGeneral} (or, equivavalently, from Note \ref{note2}) for the formulas (\ref{n5}), (\ref{n6}) and obtain basis-free formulas (generalized Vieta's formulas) for $C_{(2)}$, \ldots, $C_{(7)}$. These formulas are presented in \cite{Abd2}.
For example, in the case $n=6$ we get from (\ref{n6}):
\begin{eqnarray}
C_{(2)}&=&-((UU^{\bigtriangleup}+
            U\widetilde{U}+
            (\widehat{U}\widehat{\widetilde{U}})^{\bigtriangleup}+
            \widetilde{U}\widehat{\widetilde{U}}+
            \widehat{U}\widehat{\widetilde{U}}+
            U\widehat{U}^{\bigtriangleup} +\widetilde{U}\widehat{U}^{\bigtriangleup}+
            U\widehat{\widetilde{U}}+
            \widetilde{U}\widetilde{U}^{\bigtriangleup}
            \nonumber\\
          && +
            (U\widetilde{U})^{\bigtriangleup}+ U\widehat{U}+
            \widetilde{U}\widehat{\widetilde{U}}^{\bigtriangleup}+
            \widetilde{U}\widehat{U}+
            U\widetilde{U}^{\bigtriangleup}+
            U\widehat{\widetilde{U}}^{\bigtriangleup}+
            \widetilde{U}U^{\bigtriangleup})+
            \frac{1}{3}(
            (\widehat{\widetilde{U}}\widetilde{U})^{\bigtriangleup}
            \nonumber\\
        &&+(\widehat{\widetilde{U}}U)^{\bigtriangleup}+(\widehat{U}\widetilde{U})^{\bigtriangleup}+
            (\widehat{U}U)^{\bigtriangleup}+
            \widehat{\widetilde{U}}\widetilde{U}^{\bigtriangleup}+
            \widehat{\widetilde{U}}U^{\bigtriangleup}
            \nonumber +
            \widehat{\widetilde{U}}\widehat{\widetilde{U}}^{\bigtriangleup}+
            \widehat{\widetilde{U}}\widehat{U}^{\bigtriangleup}+
            \widehat{U}\widetilde{U}^{\bigtriangleup}\\
    &&+
            \widehat{U}U^{\bigtriangleup}  +
            \widehat{U}\widehat{\widetilde{U}}^{\bigtriangleup} +\widehat{U}\widehat{U}^{\bigtriangleup})
            +\frac{2}{3}( \widehat{\widetilde{U}}^{\bigtriangleup}\widetilde{U}^{\bigtriangleup}+\widehat{\widetilde{U}}^{\bigtriangleup}U^{\bigtriangleup}+
            \widehat{U}^{\bigtriangleup}\widetilde{U}^{\bigtriangleup}+
            \widehat{U}^{\bigtriangleup}U^{\bigtriangleup} \nonumber\\
      &&+
            (\widehat{\widetilde{U}}^{\bigtriangleup}\widetilde{U})^{\bigtriangleup}+
            (\widehat{\widetilde{U}}^{\bigtriangleup}U)^{\bigtriangleup}+(\widehat{\widetilde{U}}^{\bigtriangleup}\widehat{\widetilde{U}})^{\bigtriangleup} +(\widehat{\widetilde{U}}^{\bigtriangleup}\widehat{U})^{\bigtriangleup}+
            (\widehat{U}^{\bigtriangleup}\widetilde{U})^{\bigtriangleup}+
            (\widehat{U}^{\bigtriangleup}U)^{\bigtriangleup}\nonumber\\
            &&+
            (\widehat{U}^{\bigtriangleup}\widehat{\widetilde{U}})^{\bigtriangleup}+
            (\widehat{U}^{\bigtriangleup}\widehat{U})^{\bigtriangleup})).\nonumber
    \end{eqnarray} 
\end{example}

 \section{On alternative approaches with bar-operation}\label{sectBar}

The other three approaches (see discussions in \cite{det} and Table 1 in \cite{acus2}) to the problem of finding optimized formulas for the determinant and other characteristic polynomial coefficients in geometric algebras is to use instead of the operations $\bigtriangleup_1$, \ldots, $\bigtriangleup_m$:
\begin{enumerate}
\item one operation (we call it {\em bar-operation})
$$\overline{U}:=2\langle U \rangle_0-U;$$
\item two operations $\overline{U}$, $\widetilde{U}$ together;\\
\item three operations $\overline{U}$, $\widetilde{U}$, $\widehat{U}$ together.
\end{enumerate}
Let us introduce the parameters $t_1$, $t_2$, and $t_3$, which are equal to the minimum number of terms in these three cases. We have $t_1\geq t_2 \geq t_3 \geq 1$.

We obtain the following analogue of Theorem \ref{thCN}.

\begin{theorem}\label{th5}  For arbitrary $n\geq 1$, there exists basis-free formulas for $\Det(U)=-C_{(N)}$ of the following three types ($r=1, 2, 3$):
\begin{eqnarray}
\!\!\!\!\!\!\!\Det(U)=U (\sum_{j=1}^{t_r} \alpha_{jr} H_{jr}(U))=(\sum_{j=1}^{t_r} \alpha_{jr} H_{jr}(U)) U,\quad \alpha_{jr}\in{\mathbb R},\quad \sum_{j=1}^{t_r} \alpha_{jr}=1,\label{type123}
\end{eqnarray}
where each expression $H_{jr}(U)$ involves only multiplication (consists of $N-1$ multipliers) and:
\begin{enumerate}
\item the operation $\overline{U}$ in the case $r=1$;
\item the operations $\overline{U}$, $\widetilde{U}$ in the case $r=2$;
\item the operations $\overline{U}$, $\widetilde{U}$, $\widehat{U}$ in the case $r=3$.
\end{enumerate}
\end{theorem}
\begin{proof}  The proof in all three cases is analogous to the proof of Theorem \ref{thCN}. We use the Faddeev--LeVerrier algorithm (\ref{FL}) and the formula
$$
\langle U \rangle_0=\frac{1}{2}(U+\overline{U}).
$$
Using the classical operations $\widetilde{U}$ and  $\widehat{U}$ sometimes simplifies (the number of terms can be less $t_1\geq t_2 \geq t_3$) the formulas that contain only the operation~$\overline{U}$. \qed
\end{proof}

Note that the formulas for the determinant using only one operation $\overline{U}$ (for $r=1$ in Theorem \ref{th5}) must coincide in the cases $n=2j-1$ and $n=2j$ for all $j\geq 1$ due to the Faddeev--LeVerrier algorithm, see~\cite{det}.

\medskip

We have the following known formulas for the determinant in the cases $n\leq 6$:
\begin{eqnarray}
n=1,\qquad &&\Det(U)=U\widetilde{U}=U \overline{U},\\
n=2,\qquad &&\Det(U)=U\widehat{\widetilde{U}} =U\overline{U},\\
n=3,\qquad &&\Det(U)=U\widehat{U}\widetilde{U}\widehat{\widetilde{U}}=\frac{1}{3}U U \overline{UU}+\frac{2}{3} U \overline{ \overline{U}\, \overline{\overline{U}\, \overline{U}}}=U\widetilde{U} \overline{U \widetilde{U}},\\
n=4,\qquad &&\Det(U)=U\widehat{\widetilde{U}}(\widehat{U}\widetilde{U})^\bigtriangleup=\frac{1}{3}U U \overline{UU}+\frac{2}{3} U \overline{ \overline{U}\, \overline{\overline{U}\, \overline{U}}}=U\widetilde{U} \overline{U \widetilde{U}},\\
n=5,\qquad &&\Det(U)=U\widehat{\widetilde{U}}\widehat{U}\widetilde{U}(\widehat{U}\widetilde{U}U\widehat{\widetilde{U}})^\bigtriangleup=\frac{1}{3}H H \overline{H H}+\frac{2}{3} H \overline{ \overline{H}\, \overline{\overline{H}\, \overline{H}}}\\
&&=J \widehat{J} \overline{J \widehat{J}},\\
n=6,\qquad &&\Det(U)=\frac{1}{3}H\widehat{H}(\widehat{H} H)^\bigtriangleup+\frac{2}{3}H(\widehat{H}^\bigtriangleup(\widehat{H}^\bigtriangleup H^\bigtriangleup)^\bigtriangleup)^\bigtriangleup\\
&&=\frac{1}{3}H H \overline{H H}+\frac{2}{3} H \overline{ \overline{H}\, \overline{\overline{H}\, \overline{H}}},
\end{eqnarray}
where we use the notation $J:=U \widehat{\widetilde{U}}$, $H:=U\widetilde{U}$. 

We obtain the following Table~1 with known minimal values of the parameters $t$, $t_1$, $t_2$, $t_3$ (for the parameter $t$, see Theorem \ref{thCN}). We have $t_3=2$ ($t_3\neq 1$) in the case $n=6$ according to the results of the paper \cite{acus}, where grade-negation operations of different types are used. 

The question mark stands at those places where the minimum values of the corresponding parameters are unknown. An open problem is to find optimized formulas with the parameter $t_1\geq 2$ in the cases $n=5, 6$ (these two formulas must coincide, see above). An open problem is to find optimized formulas of all four types with the parameters $t$, $t_1$, $t_2$, $t_3$ for the determinant in geometric algebras in the case of high dimensions $n\geq 7$.

\begin{table}[ht]\begin{center}
\begin{tabular}{|c|c|c|c|c|c|c|c|} \hline
$\quad n \quad$   & $\quad 1\quad$ & $\quad2\quad$ & $\quad3\quad$ & $\quad4\quad$ & $\quad5\quad$ & $\quad6\quad$\\  \hline 
$t$   & 1 & 1 & 1 & 1 & 1 & 2 \\ \hline
$t_1$ & 1 & 1 & 2 & 2 & ? & ? \\ \hline
$t_2$ & 1 & 1 & 1 & 1 & 2 & 2 \\ \hline
$t_3$ & 1 & 1 & 1 & 1 & 1 & 2 \\ \hline
\end{tabular}\end{center}
\label{table1}\caption{The values of the parameters $t$, $t_1$, $t_2$, $t_3$.}
\end{table}

We get the following analogue of Theorem \ref{thGeneral} (generalized Vieta theorem) for all three types of formulas.

\begin{theorem}\label{th6} Let us consider an arbitrary element $U\in\cl_{p,q}$, $n=p+q$, $N=2^{[\frac{n+1}{2}]}$ with the explicit formula for $\Det(U)=-C_{(N)}$ of the types (\ref{type123}). Let us formally change all expressions $\overline{U_1 \cdots U_j}$ to $\overline{U_1}\cdots \overline{U_j}$ in each $H_{jr}(U)$ separately. Let us write down elementary symmetric polynomials for these expressions separately. Let us formally change back previously considered expressions $\overline{U_1}\cdots \overline{U_j}$  to $\overline{U_1 \cdots U_j}$ in these formulas. We obtain the right formulas for the coefficients $(-1)^{k+1} C_{(k)}$, $k=1, \ldots, N-1$, which we call the generalized Vieta's formulas in Clifford geometric algebras.
\end{theorem}

\begin{proof} The proof is analogous to the proof of Theorem \ref{thGeneral}. We use that the operation $\overline{U}$ is linear and does not change the identity element $e$. \qed
\end{proof}

\begin{note}
Alternatively (analogously to Note 2), we can use the following equivalent algorithm. Let us consider an arbitrary element $U\in\cl_{p,q}$, $n=p+q$, $N=2^{[\frac{n+1}{2}]}$ with the explicit formula for $C_{(N)}=-\Det(U)$ of the types (\ref{type123}). Then all other characteristic polynomial coefficients are
    \begin{eqnarray}
        C_{(k)}=(-1)^{k+1}\sum_{X_l\in X(k)}{\rm F}(X_l),\quad
        k = 1,2,\ldots,N,
    \end{eqnarray}
    where ${\rm F}$ is $N$-variable function obtained from (\ref{type123}) by enumerating every occurrence of the element $U$ from left to right (in each of $t_r$ terms respectively), and $X(k)$ is the set of all possible tuples with $k$ elements $U$ and $N-k$ identity elements $e$:
    \begin{eqnarray}
        \!\!\!\!\!X(k)=\left\lbrace X_l=(x_1,x_2,\ldots,x_N) \, | \,
                    x_i\in\lbrace e,U\rbrace, \quad \sum_{i=1}^N x_i=kU+(N-k)e
                \right\rbrace.
    \end{eqnarray}
\end{note}

\begin{example} In the cases $n=1, 2$, from the formula
\begin{eqnarray}
\Det(U)=-C_{(2)}=U\overline{U},
\end{eqnarray}
we obtain
\begin{eqnarray}
C_{(1)}=U+\overline{U}.
\end{eqnarray}
In the cases $n=3, 4$, from the formula 
\begin{eqnarray}
\Det(U)=-C_{(4)}=\frac{1}{3}U U \overline{UU}+\frac{2}{3} U \overline{ \overline{U}\, \overline{\overline{U}\, \overline{U}}},
\end{eqnarray}
we obtain
\begin{eqnarray}
C_{(1)}&=&2(U+\overline{U}),\\
-C_{(2)}&=&UU+\frac{1}{3}\overline{UU}+\frac{2}{3}\overline{U}\,\overline{U}+\frac{8}{3}U\overline{U}+\frac{4}{3}\overline{\overline{U}U},\\
C_{(3)}&=&\frac{2}{3}U\overline{UU}+\frac{2}{3}UU\overline{U}+\frac{4}{3}U\overline{\overline{U}U}+\frac{2}{3}\overline{\overline{U}\,\overline{\overline{U}\,\overline{U}}}+\frac{2}{3}U\overline{U}\,\overline{U}.
\end{eqnarray}
In the cases $n=3, 4$, from the formula 
\begin{eqnarray}
\Det(U)=-C_{(4)}=U\widetilde{U} \overline{U \widetilde{U}},
\end{eqnarray}
we obtain
\begin{eqnarray}
C_{(1)}&=&U+\widetilde{U}+\overline{U}+\overline{\widetilde{U}},\\
-C_{(2)}&=&\overline{U \widetilde{U}}+\widetilde{U}\overline{\widetilde{U}}+\widetilde{U}\overline{U}+U\overline{\widetilde{U}}+U\overline{U}+U\widetilde{U},\\
C_{(3)}&=&\widetilde{U}\overline{U\widetilde{U}}+U\overline{U\widetilde{U}}+U\widetilde{U}\overline{\widetilde{U}}+U\widetilde{U}\overline{U}.
\end{eqnarray}
In the case $n=5$, we can get the right formulas for all characteristic polynomial coefficients using the same algorithm and the formula
\begin{eqnarray}\Det(U)=-C_{(8)}=J \widehat{J} \overline{J \widehat{J}}=U \widehat{\widetilde{U}} \widehat{U} \widetilde{U}\overline{U \widehat{\widetilde{U}} \widehat{U} \widetilde{U}}.
\end{eqnarray}
For example, we get
\begin{eqnarray}
C_{(1)}&=&U +\widehat{\widetilde{U}} +\widehat{U} +\widetilde{U}+\overline{U} +\overline{\widehat{\widetilde{U}}}+\overline{\widehat{U}}+\overline{\widetilde{U}},\\
-C_{(2)}&=&U \widehat{\widetilde{U}}+U \widehat{U}+U \widetilde{U}+U \overline{U}+U \overline{ \widehat{\widetilde{U}}}+U \overline{\widehat{U}}+U \overline{ \widetilde{U}}+\widehat{\widetilde{U}} \widehat{U}+ \widehat{\widetilde{U}}\widetilde{U}+\widehat{\widetilde{U}} \overline{U}\\
&+&\widehat{\widetilde{U}}\overline{\widehat{\widetilde{U}}}+\widehat{\widetilde{U}} \overline{\widehat{U}}+\widehat{\widetilde{U}}\overline{\widetilde{U}}+\widehat{U} \widetilde{U}+\widehat{U} \overline{U}+\widehat{U} \overline{\widehat{\widetilde{U}}}+\widehat{U} \overline{ \widehat{U} }+\widehat{U} \overline{\widetilde{U}}+\widetilde{U}\overline{U}+\widetilde{U}\overline{\widehat{\widetilde{U}}}\nonumber\\
&+&\widetilde{U}\overline{\widehat{U}}+\widetilde{U}\overline{\widetilde{U}}+\overline{U \widehat{\widetilde{U}}}+\overline{U \widehat{U} }+\overline{U \widetilde{U}}+\overline{ \widehat{\widetilde{U}} \widehat{U} }+\overline{\widehat{\widetilde{U}} \widetilde{U}}+\overline{ \widehat{U} \widetilde{U}},\nonumber\\
C_{(7)}&=&\widehat{\widetilde{U}} \widehat{U} \widetilde{U}\overline{U \widehat{\widetilde{U}} \widehat{U} \widetilde{U}}+U \widehat{U} \widetilde{U}\overline{U \widehat{\widetilde{U}} \widehat{U} \widetilde{U}}+U \widehat{\widetilde{U}} \widetilde{U}\overline{U \widehat{\widetilde{U}} \widehat{U} \widetilde{U}}+U \widehat{\widetilde{U}} \widehat{U} \overline{U \widehat{\widetilde{U}} \widehat{U} \widetilde{U}}\\
&+&U \widehat{\widetilde{U}} \widehat{U} \widetilde{U}\overline{ \widehat{\widetilde{U}} \widehat{U} \widetilde{U}}+U \widehat{\widetilde{U}} \widehat{U} \widetilde{U}\overline{U  \widehat{U} \widetilde{U}}+U \widehat{\widetilde{U}} \widehat{U} \widetilde{U}\overline{U \widehat{\widetilde{U}} \widetilde{U}}+U \widehat{\widetilde{U}} \widehat{U} \widetilde{U}\overline{U \widehat{\widetilde{U}} \widehat{U} }.\nonumber
\end{eqnarray}
Similarly, we can write down explicit formulas for the characteristic polynomial coefficients $C_{(3)}$, $C_{(4)}$, $C_{(5)}$, $C_{(6)}$. We do not do this here because of their cumbersomeness.

In the cases $n=5, 6$, we can get the right formulas of another type for all characteristic polynomial coefficients using the same algorithm and the formula
\begin{eqnarray}\Det(U)=-C_{(8)}=\frac{1}{3}U\widetilde{U} U\widetilde{U} \overline{U\widetilde{U} U\widetilde{U}}+\frac{2}{3} U\widetilde{U} \overline{ \overline{U\widetilde{U}}\, \overline{\overline{U\widetilde{U}}\, \overline{U\widetilde{U}}}}.
\end{eqnarray}
For example, we get
\begin{eqnarray}
C_{(1)}&=&2(U+\widetilde{U}+\overline{U}+\overline{\widetilde{U}}),\\
C_{(2)}&=&3U\widetilde{U}+UU+\widetilde{U} U+\widetilde{U} \widetilde{U}+\frac{8}{3}(U \overline{U}+U\overline{\widetilde{U}}+\widetilde{U}\overline{U}+\widetilde{U}\overline{\widetilde{U}})+\frac{7}{3}\overline{U\widetilde{U}}\\
&+&\frac{4}{3}(\overline{\overline{U}U}+\overline{\overline{U}\widetilde{U}}+\overline{\overline{\widetilde{U}}U}+ \overline{\overline{\widetilde{U}}\widetilde{U}})+\frac{2}{3}(\overline{U}\,\overline{U}+\overline{U}\,\overline{\widetilde{U}}+\overline{\widetilde{U}}\,\overline{U}+\overline{\widetilde{U}}\,\overline{\widetilde{U}})\nonumber\\
&+& \frac{1}{3}(\overline{UU}+\overline{\widetilde{U} U}+\overline{\widetilde{U}\widetilde{U}}),\nonumber\\
C_{(7)}&=&\frac{1}{3}(\widetilde{U} U\widetilde{U} \overline{U\widetilde{U} U\widetilde{U}}+U U\widetilde{U} \overline{U\widetilde{U} U\widetilde{U}}+U\widetilde{U} \widetilde{U} \overline{U\widetilde{U} U\widetilde{U}}+U\widetilde{U} U\overline{U\widetilde{U} U\widetilde{U}}\\
&+&U\widetilde{U} U\widetilde{U} \overline{\widetilde{U} U\widetilde{U}}+U\widetilde{U} U\widetilde{U} \overline{U U\widetilde{U}}+U\widetilde{U} U\widetilde{U} \overline{U\widetilde{U} \widetilde{U}}+U\widetilde{U} U\widetilde{U} \overline{U\widetilde{U} U})\nonumber\\
&+&\frac{2}{3} (\widetilde{U} \overline{ \overline{U\widetilde{U}}\, \overline{\overline{U\widetilde{U}}\, \overline{U\widetilde{U}}}}+U\overline{ \overline{U\widetilde{U}}\, \overline{\overline{U\widetilde{U}}\, \overline{U\widetilde{U}}}}+U\widetilde{U} \overline{ \overline{\widetilde{U}}\, \overline{\overline{U\widetilde{U}}\, \overline{U\widetilde{U}}}}
+U\widetilde{U} \overline{ \overline{U}\, \overline{\overline{U\widetilde{U}}\, \overline{U\widetilde{U}}}}\nonumber\\
&
+&U\widetilde{U} \overline{ \overline{U\widetilde{U}}\, \overline{\overline{\widetilde{U}}\, \overline{U\widetilde{U}}}}
+U\widetilde{U} \overline{ \overline{U\widetilde{U}}\, \overline{\overline{U}\, \overline{U\widetilde{U}}}}
+U\widetilde{U} \overline{ \overline{U\widetilde{U}}\, \overline{\overline{U\widetilde{U}}\, \overline{\widetilde{U}}}}
+U\widetilde{U} \overline{ \overline{U\widetilde{U}}\, \overline{\overline{U\widetilde{U}}\, \overline{U}}})\nonumber
\end{eqnarray}
Similarly, we can write down explicit formulas for the characteristic polynomial coefficients $C_{(3)}$, $C_{(4)}$, $C_{(5)}$, $C_{(6)}$. We do not do this here because of their cumbersomeness.
\end{example}

\section{Conclusions}\label{sectConcl}

In this paper, we discuss a generalization of Vieta's formulas to the case of (Clifford) geometric algebras. We apply the Gelfand--Retakh theorem to the characteristic polynomial in geometric algebras. We show how to express characteristic coefficients in terms of combinations of various involutions of elements. We compare the generalized Vieta's formulas with the ordinary Vieta's formulas for eigenvalues. The role of the roots (which are complex scalars) of the characteristic equation is played by some combinations of involutions of elements (which are not scalars). The cases of small dimensions $n\leq 3$ are discussed in details. The case of arbitrary eigenvalues (including the case of degenerate eigenvalues) is considered. We introduce the notion of a simple basis-free formula for the determinant in geometric algebra and prove that a formula of this type exists in the case of arbitrary dimension. Using this notion, we present and prove generalized Vieta theorem in geometric algebra of arbitrary dimension. In this theorem, we obtain explicit formulas for all characteristic polynomial coefficients from the highest one, which is the determinant. Alternative three approaches to finding optimized formulas for the determinant and other characteristic polynomial coefficients using the bar-operation are discussed.

We hope that the new approaches presented in this paper (related to noncommutative symmetric functions and generalized Vieta's formulas) will help to find more optimized (simple) formulas for the determinant and inverse in geometric algebras in the cases $n\geq 7$. We thank one of the anonymous reviewers of this paper for pointing out possible applications of the results of this work in theoretical and mathematical physics, in particular, in (quantum) field theory on pseudo-Riemannian manifolds and the Yang--Mills theory.

\subsubsection{Acknowledgements} The publication was prepared within the framework of the Academic Fund Program at the HSE University in 2022 (grant 22-00-001). 

The results of this paper were reported at the International Conference
Computer Graphics International 2022 (CGI2022), Empowering Novel Geometric Algebra for Graphics \& Engineering Workshop (Geneva, Switzerland, September 2022). The author is grateful to the organizers and the participants of this conference for fruitful discussions.

The author is grateful to K. Abdulkhaev and N. Marchuk for useful discussions. The author is grateful to the anonymous reviewers for their careful reading of the paper and helpful comments on how to improve the presentation.

\noindent{\bf Data availability} Data sharing not applicable to this article as no datasets
were generated or analyzed during the current study.

\noindent{\bf Conflict of interest}
This work does not have any conflicts of interest.


\end{document}